\documentclass[reqno]{amsart}
\usepackage[T1]{fontenc}
\usepackage{amssymb, amsmath, color, textcomp, url,tikz}
\usepackage{tkz-tab}
\usepackage{caption}
\usetikzlibrary{matrix,arrows}
    \usetikzlibrary{tikzmark, decorations.pathreplacing}

\usetikzlibrary{fit}
\usetikzlibrary{positioning}
\usepackage[labelformat=empty]{caption}
\usepackage{mdwlist}
\usepackage{etoolbox}

\RequirePackage{ifpdf}
\ifpdf
   \usepackage[pdftex]{hyperref}
\else
   \usepackage[hypertex]{hyperref}
\fi

\usepackage{enumerate,xspace}
\usepackage{subcaption}

\theoremstyle{plain}
\newtheorem{theorem}{Theorem}[section]
\newtheorem{cor}[theorem]{Corollary}
\newtheorem{prop}[theorem]{Proposition}
\newtheorem{lemma}[theorem]{Lemma}

\newtheorem*{claimstern}{Claim}
\newenvironment{claimsternproof}{\noindent\textit{Proof of
    Claim.}}{\hfill\qedsymbol \tiny{ Claim}
\medskip}

\theoremstyle{definition}
\newtheorem{remark}[theorem]{Remark}
\newtheorem{fact}[theorem]{Fact}
\newtheorem{definition}[theorem]{Definition}

\newtheorem*{question}{Question}
\newtheorem*{notation}{Notation}

\newcommand{\nc}{\newcommand}

\nc{\Z}{\mathbb{Z}}
\nc{\N}{\mathbb{N}}
\nc\LL{\mathcal L}

\renewcommand{\phi}{\varphi}

\nc\ord{\operatorname{ord}}
\nc{\acl}{\operatorname{acl}}
\nc{\aclq}{\operatorname{acl^{eq}}}
\nc\inv{ ^{-1}}

\nc{\tp}{\operatorname{tp}}
\nc{\cf}{\text{cf. }}
\nc\cl{\operatorname{icl}}
\nc\icl{indiscernibly closed\xspace}
\nc\icyc{indiscernibly acyclic\xspace}
\nc\pf{\rightarrow} 
\nc\Ker{\operatorname{Ker}}
\nc{\restr}[1]{\!\!\upharpoonright_{#1}}
\nc{\mi}{\mathrm{M}_\infty}  
\nc{\lmto}[1]{\xrightarrow[{#1}]{}} 
\nc{\dd}{\operatorname{d}} 
\nc{\ps}{\mathrm{PS}} 
\nc{\cps}{\mathrm{CPS}} 

\nc{\fcl}[1]{\langle #1 \rangle}
\nc{\EP}{\mathrm{EP}} 
\nc{\expt}{\mathrm{ep}} 
\nc{\de}{\mathrm{defect}} 

\def\Ind#1#2{#1\setbox0=\hbox{$#1x$}\kern\wd0\hbox to 0pt{\hss$#1\mid$\hss}
\lower.9\ht0\hbox to 0pt{\hss$#1\smile$\hss}\kern\wd0}
\def\Notind#1#2{#1\setbox0=\hbox{$#1x$}\kern\wd0\hbox to
0pt{\mathchardef\nn="0236\hss$#1\nn$\kern1.4\wd0\hss}\hbox
to 0pt{\hss$#1\mid$\hss}\lower.9\ht0
\hbox to 0pt{\hss$#1\smile$\hss}\kern\wd0}
\def\ind{\mathop{\mathpalette\Ind{}}}

\def\ld{\mathop{\ \ \hbox to 0pt{\hss$\mid^{\hbox to
0pt{$\scriptstyle\mathrm{ld}$\hss}}$\hss}
\lower4pt\hbox to 0pt{\hss$\smile$\hss}\ \ }}

\begin{document}

\title{Trois couleurs: A new non-equational theory}
\date{September 14, 2020 }

\author{ Amador Martin-Pizarro and Martin Ziegler}
\address{Mathematisches Institut,
  Albert-Ludwigs-Universit\"at Freiburg, Ernst-Zermelo-Str 1,  D-79104
  Freiburg, Germany}
\email{pizarro@math.uni-freiburg.de}
\email{ziegler@uni-freiburg.de}
\thanks{ The first author conducted research partially supported
by the program GeoMod AAPG2019 (ANR-DFG). Both authors were
supported by the program MTM2017-86777-P}

\begin{abstract}
  A first-order theory is equational if every definable set is a
  Boolean combination of instances of equations, that is, of formulae
  such that the family of finite intersections of instances has the
  descending chain condition. Equationality is a strengthening of
  stability yet so far only two examples of non-equational stable
  theories are known. We construct non-equational $\omega$-stable theories by
  a suitable colouring of the free pseudospace, based on Hrushovski
  and Srour's original example.
\end{abstract}

\maketitle

\section{Introduction}

Consider a first order complete theory $T$. A formula $\phi(x;y)$ is
an \emph{equation} (for a given partition of the free variables into
$x$ and $y$) if, in every model of $T$, the family of finite
intersections of instances $\phi(x,a)$ has the descending chain
condition. The theory $T$ is \emph{equational} if every formula
$\psi(x;y)$ is equivalent modulo $T$ to a Boolean combination of
equations $\phi(x;y)$.

Determining whether a particular stable theory is equational is not
obvious. So far, the only known \emph{natural} example of a stable
non-equational theory is the free non-abelian group
\cite{zS13, MS17}, though the first example of a non-equational stable
theory is of combinatorial nature and appeared in unpublished notes of
Hrushovski and Srour \cite{HS89}. They coloured the free pseudospace
\cite{BP00} with two colours in order to obtain two types $r(x,y)\neq
r'(x,y)$ which are not \emph{equationally separated}, according to the
termino\-logy of \cite[Section 2.1]{eH12}, that is, there are
sequences $(a_i, b_i)_{i\in \mathbb{N}}$ and $(c_i, d_i)_{i\in
  \mathbb{N}}$, which can be taken indiscernible over $\emptyset$,
such that $r(a_i,b_i)$ and $r'(c_i, d_i)$ holds for all $i$, but
$r'(a_i, b_j)$ and $r(c_i, d_j)$ holds for $i<j$. In an equational
theory, any two distinct types are equationally separated.

In all previously known examples of non-equational theories the
failure of equationality is due
to the presence of two distinct non-equationally separated types
$r(x,y)\neq r'(x,y)$ such that the length of $x$ is $1$. In this note,
we will build on Hrushovski-Srour's example in order to construct new
examples of non-equational theories, where all distinct real types
$p\neq q$ in finitely many variables are equationally separated.
\section{Equations and indiscernibly closed sets}

Most of the results in this section come from \cite{PiSr84, mJdL01}.

Consider a first order theory $T$. A formula $\phi(x;y)$ is an
\emph{equation} (with respect to a given partition of the free
variables into $x$ and $y$) if, in every model of $T$, the family of
finite intersections of instances $\phi(x,b)$ has the descending chain
condition. An easy compactness argument shows

\begin{lemma}\label{L:eq_seq}
  The formula $\phi(x;y)$ is an equation if there is no sequence
  $(a_i,b_i)_{i\in\N}$ in any model $M$ such that
  $M\models\phi(a_i,b_j)$ and $M\not\models\phi(a_i,b_i)$ for all
  $i<j$.
\end{lemma}
A Ramsey argument shows that, working in a sufficiently saturated
model, the sequence $(a_i,b_i)$ can be assumed to be indiscernible of
any infinite order type. Thus, if $\phi(x;y)$ is an equation, then so
are $\phi\inv(x;y)=\phi(y,x)$ and $\phi(f(x);y)$, whenever $f$ is a
$\emptyset$-definable function, which maps finite tuples to finite
tuples. Finite conjunctions and disjunctions of equations are again
equations. Note that equations are stable formulae.

In \cite{mJdL01}, an equivalent definition of equations was obtained
in terms of \icl sets: an element $c$ lies in the \emph{indiscernible
  closure} $\cl(X)$ of a set $X$ if there is an indiscernible sequence
$(a_i)_{i\in\N}$ such that $a_i$ lies in $X$ for $i >0$ and $a_0=c$.
Note that $X\subset \cl(X)$. A set $X$ is \emph{\icl} if $X=\cl(X)$.
\begin{lemma}\cite[Theorem 3.16]{mJdL01}\label{L:eq_icl}
  A formula $\phi(x;y)$ is an equation if and only if the set
  $\phi(M,b)$ is \icl in in every model $M$ of $T$.
\end{lemma}
\begin{proof}
  Let us work inside a sufficiently saturated model $M$. If
  $\phi(x;y)$ is not an equation, witnessed by the indiscernible
  sequence $(a_i,b_i)_{i\in\Z}$, as in Lemma \ref{L:eq_seq}, the set
  defined by $\phi(x,b_0)$ is not \icl, for it contains all $a_i$'s
  with $i<0$, but does not contain $a_0$. Conversely, if some instance
  $\phi(x,b)$ is not \icl, there is an indiscernible sequence
  $(a_i)_{i\in\Z}$ such that $M\models\phi(a_i,b)$ for $i<0$, but
  $M\not\models\phi(a_0,b)$. For every $j$ in $\mathbb Z$, there is an
  element $b_j$ in $M$ such that $M\models\phi(a_i,b_j)$ for $i<j$,
  but $M\not\models\phi(a_j,b_j)$.
\end{proof}

The theory $T$ is \emph{equational} if every formula $\psi(x;y)$ is
equivalent modulo $T$ to a Boolean combination of equations
$\phi(x;y)$. Since Boolean combinations of stable formulas are stable,
equational theories are stable.

Typical examples of equational theories are the theory of an
equivalence relation with infinite many infinite classes, the theory
of $R$-modules for some ring $R$, or the theory of algebraically
closed fields.

Equationality is preserved under bi-interpretability as well as
addition of  parameters \cite{Ju00}. It is unknown whether
equationality
holds if every formula $\phi(x;y)$, with $x$ a single variable, is a
boolean combination of equations.

It is easy to see that $T$ is equational if and only if all
completions of $T$ are equational. So for the rest of this section we
assume that $T$ is complete and work in a sufficiently saturated model
$\mathbb U$.

Notice that a theory $T$ is equational if and only if every type $p$
over $A$ is implied by its \emph{equational part} $\{\varphi(x,a) \in
p\mid \varphi(x;y) \text{ is an equation} \}$.

\begin{definition}\label{D:Pfeile}
  Given two types $p(x,b)$ and $q(x,b)$, define $p(x,b)\pf q(x,b)$ if
  $q(x,b) \subset \cl(p(x,b))$, or equivalently, if there is an
  indiscernible sequence $(a_i)_{i\in\N}$ such that all $\models
  p(a_i,b)$ for $i>0$ and $\models q(a_0,b)$. If $p(x,y)$ and $q(x,y)$
  are the the corresponding (complete) types over $\emptyset$, we
  write
\[p(x;y)\pf q(x;y).\]
\end{definition}
A standard argument as in Lemma \ref{L:eq_icl} with $p$ instead of
$\phi$ and and $q$ instead of $\neg\phi$ yields the following:
\begin{lemma}\label{L:pf_seq}
  We have that $p(x;y)\pf q(x;y)$ if an only if there is a sequence
  $(a_i,b_i)_{i\in\N}$ such that $\models p(a_i,b_j)$ for $i<j$, and
  $\models q(a_i,b_i)$ for all $i$. Furthermore, we may assume that
  the sequence is indiscernible and of any given infinite order type.
\end{lemma}

The above characterisation provides an easy proof of the following remark:
\begin{remark}\label{R:pf_alg}
  Clearly $p\pf p$. If $p(x;y)\pf q(x;y)$, then $p^{-1}\pf q^{-1}$,
  where $p^{-1}(x;y)=p(y;x)$.

  Furthermore, if $\tp(a;b)\pf \tp(a';b)$, then $a\stackrel{\mathrm
    stp}{\equiv} a'$. Thus, if $p(x;y)$ implies that $x$ (or $y$) is
  algebraic, then $p\pf q$ only when $q=p$.
\end{remark}

\begin{cor}\label{C:pf_f}
  Let $f$ and $g$ be $\emptyset$-definable functions and $a,a',b,b'$
  finite tuples, with $\tp(a;b)\pf\tp(a';b')$. Then
  $\tp(f(a);g(b))\pf\tp(f(a');g(b'))$.
\end{cor}

\begin{cor}\label{C:pf_cl}
  A formula $\varphi(x;y)$ is an equation if and only if, whenever a
  type $p(x, y)$ contains $\varphi(x,y)$ and $p(x;y)\pf q(x;y)$, then
  $\varphi(x,y)$ lies in $q(x;y)$.
\end{cor}

\begin{proof}
  One direction follows clearly from Lemma \ref{L:eq_icl}. For the
  converse, assume that $\phi(x;y)$ is not an equation and choose an
  indiscernible sequence $(a_i,b_i)_{i\in\N}$ as in Lemma
  \ref{L:eq_seq}. Let $p$ be the common type of the pairs $(a_i,b_j)$,
  with $i<j$ and $q$ be the common type of the pairs $(a_i,b_i)$. Then
  $p\pf q$ and $\phi$ belongs to $p$, but not to $q$.
\end{proof}

\begin{definition}\label{D:iclacyclic}
  A \emph{cycle of types} is a sequence
  \[ p_0(x;y)\pf p_1(x;y)\pf\dotsb \pf
  p_{n-1}(x;y)\pf p_0(x;y).\] The cycle is \emph{proper} if all the
  $p_i$'s are different. The theory $T$ is \emph{\icyc} if there is no
  proper cycle of types of length $n\geq 2$.
\end{definition}

Following the terminology of \cite[Section 2.1]{eH12}, two distinct
types $p(x;y)$ and $q(x;y)$ are not \emph{equationally separated} if
and only $p\pf q\pf p$.

\begin{remark}\label{R:icon_st}
  Every \icyc theory is stable.
\end{remark}

\begin{proof}
  If there is a formula $\phi(x;y)$ in $T$ with the order property,
  find an indiscernible sequence $(a_i)_{i\in \Z}$ in $\mathbb U$ such
  that $\models \phi(a_i, a_j) \text{ if and only if } i<j$. Set
  $p=\tp(a_1;a_0)$ and $q=\tp(a_{-1};a_0)$. Then $p\neq q$, and since
  the sequence $(a_i)_{i\not=0}$ is indiscernible, we have that $p\pf
  q\pf p$, so there is a proper cycle of types of length $2$.
\end{proof}

\begin{remark}\label{R:eq_icon}
  Every equational theory is \icyc.
\end{remark}

\begin{proof}
  Consider a cycle \[ p_0\pf p_1\pf\dotsb \pf p_{n-1} \pf p_0.\] By
  Corollary \ref{C:pf_cl}, all the types $p_i$ contain the same
  equations, so they all agree, by equationality of $T$.
\end{proof}

\begin{definition}\label{D:MS}
  The theory $T$ \emph{satisfies the MS-criterion} if there is some
  formula $\varphi(x,y)$ and a matrix $(a_{ij},b_{ij})_{i, j\in \N}$
  such that:
  \begin{enumerate}
  \item\label{D:MS:phi} $\models \varphi(a_{ij},b_{il})$ if and only
    if $j=l$.
  \item\label{D:MS:iso} $a_{ij},b_{ij}\equiv a_{ij},b_{kl}$, whenever
    $i<k$ and $j<l$.
  \end{enumerate}
\end{definition}

\begin{lemma}\label{L:MS_noneq}
  If a theory $T$ satisfies the MS-criterion, then there is a proper
  cycle of types $p\pf q\pf p$. In particular, the theory is not
  equational (\cf \cite[Proposition 2.6]{MS17}).
\end{lemma}

\begin{proof}
  We may assume that the matrix $(a_{ij},b_{ij})_{i,j \in \N}$ is
  indiscernible, that is, the type $\tp(a_{ij},b_{ij})_{i\in I, j\in
    J}$ only depends on $|I|$ and $|J|$. Set $p=\tp(a_{00};b_{00})$,
  $q=\tp(a_{00};b_{01})$ and $r=\tp(a_{00};b_{11})$. Since
  $(a_{0j}b_{0j})_{j\in \N}$ is indiscernible, we have $q\pf p$. Since
  $(a_{i0}b_{i1})_{i\in\N}$ is indiscernible, we have $r\pf q$.

  Now, by Definition \ref{D:MS} $(\ref{D:MS:phi})$, the formula
  $\varphi(x;y)$ belongs to $p(x;y)$ but not to $q(x;y)$, so $p\neq
  q$. By Definition \ref{D:MS} $(\ref{D:MS:iso})$ we have $p=r$, as desired.
\end{proof}
\noindent Since $p$ and $q$ contain the same equations, it follows
that the above formula $\phi$ cannot be a boolean combination of equations (cf.
\cite[Proposition 2.6]{MS17}).

Let us assume for the rest of this section that $T$ is stable.

\begin{lemma}\label{L:erweiterung}
  Let $p_0(x,b)\pf\dotsb \pf p_{n-1} (x,b)\pf p_0(x,b)$ be a proper
  cycle of types and $b'$ be some tuple such that $p_0(x,b)$ has only
  finitely many distinct non\-forking extensions to $bb'$. Then there
  is a proper cycle of types starting with some nonforking extension
  $p'_0(x;b,b')$ of $p_0(x,b)$ whose length is a multiple of $n$.
\end{lemma}

\begin{proof}
  First notice that, whenever $p(x,b)\pf q(x,b)$ and $q'(x,b,b')$ is a
  non-forking extension of $q(x,b)$, then $p(x,b)$ has a nonforking
  extension $p'(x,b,b')$ with \[p'(x;b,b')\pf q'(x;b,b').\] Indeed,
  consider an indiscernible sequence $(a_i)_{i\in\N}$ such that
  $\models p(a_i,b)$, for $i>0$, and $q(a_0,b)$. We may assume that
  $a_0$ realises $q'(x,b,b')$ and that the sequence $(a_i)_{i\in\N}$
  is independent from $b'$ over $b$. By a Ramsey argument, we
  may assume that the sequence $(a_i)_{i>0}$ is indiscernible over
  $a_0bb'$. Set now $p'(x,b,b')$ to be the type of $a_1$ over $bb'$,
  so $p'(x,b,b')\pf q'(x,b,b')$, as desired.

  Let $k$ now be the number of distinct nonforking extensions of
  $p_0(x,b)$ to $bb'$. Working backwards in the cycle of types, we
  deduce from the above that there is a sequence $r_0(x;b,b')
  \pf\dotsb\pf r_{n\cdot k}(x;b,b')$, where $r_{n\cdot i+j}(x,b,b')$
  is a non-forking extension of $p_j(x,b)$ for each $i\leq k$. Since
  $p_0$ has only finitely many distinct nonforking extensions to
  $bb'$, there are two indices $i<i'$ such that $r_{n\cdot
    i}(x,b,b')=r_{n\cdot i'}(x,b,b')$. Choose $i$ and $i'$ such that
  $0<i'-i$ is least possible. Then \[ r_{n\cdot i}(x;b,b')\pf\dotsb \pf
  r_{n\cdot i'}(x; b,b')\] is a proper cycle of types.
\end{proof}

\begin{cor}\label{C:teq}
  If $T$ is totally transcendental, then it is \icyc if and only if so
  is $T^\mathrm{eq}$.
\end{cor}
\begin{proof}
  We need only show that $T^\mathrm{eq}$ is \icyc, provided that $T$
  is \icyc. Assume first that the type $p(x,e)$ starts a proper cycle
  of types, where $e$ is an imaginary element. Choose a real tuple $b$
  such that $\pi_E(b)=e$ for some $0$-definable equivalence relation
  $E$. Since $T$ is totally transcendental, the type $p(x,e)$ has only
  finitely many nonforking extensions to $\{b,e\}$, so there is a proper
  cycle starting with some nonforking extension $p'(x,b,e)$ of $p(x,e)$, by the
  Lemma \ref{L:erweiterung}. By the Corollary \ref{C:pf_f}, if we restrict
  the types in the cycle to $b$, we have a cycle of types which must
  be proper, because $e$ is definable from $b$.

  Since the relation $\pf$ is symmetric in $x$ and $y$, we can now
  replace $x$ by some real tuple, so $T$ is not \icyc.
\end{proof}

\begin{notation}
  Given stationary types $p_1(x,b)$ and $p_2(x',b)$,  denote by
  $p_1(x,b)\otimes p_2(x',b)$ the type of the pair $(a_1, a_2)$ over
  $b$, where $\models p_i(a_i, b)$, for $i=1, 2$, and $a_1\ind_b a_2$.

  Observe that \[ p_1(x,b)\otimes \Big(p_2(x',b) \otimes p_3(x'',
  b)\Big)= \Big(p_1(x,b)\otimes p_2(x',b) \Big)\otimes p_3(x'', b).\]
\end{notation}

\begin{lemma}\label{L:pf_tensor}
  Given stationary types $p^j(x^j,y^j,c)$ and $q^j(x^j,y^j,c)$ over a
  tuple $c$ in $\aclq(\emptyset)$ such that
  \[ p^j(x^j;y^j,c)\pf  q^j(x^j;y^j,c), \text{ for } j=1,2,\]
  then \[p^1(x^1;y^1,c)\otimes p^2(x^2;y^2,c)\pf q^1(x^1;y^1,c)\otimes
  q^2(x^2;y^2,c).\]
\end{lemma}

\noindent By the above, the lemma generalises to an arbitrary finite
product of types.

\begin{proof}
  \noindent For $j=1,2$, choose a tuple $b^j$ and an indiscernible
  sequence $(a^j_i)_{i\in\N}$ such that $\models p^j(a_i^j,b^j,c)$,
  for $i>0$, and $\models q^j(a^j_0,b^j,c)$. We may assume that
  \[ b^1\cup \{a^1_i\}_{i\in\N} \ind_c  b^2\cup\{a^2_i\}_{i\in\N}. \]
  Since $c$ is algebraic over $\emptyset$, the sequences
  $\{a^1_i\}_{i\in\N} $ and $\{a^2_i\}_{i\in\N} $ are both
  indiscernible over $c$ and therefore mutually indiscernible, by
  stationarity of strong types, so $\{a^1_i, a^2_i\}_{i\in\N}$ is
  indiscernible. Notice that $(a_i^1, a_i^2)$ realises
  $p^1(x^1;b^1,c)\otimes p^2(x^2;b^2,c)$ for $i>0$, and $(a^1_0,
  a^2_0)$ realises $q^1(x^1;b^1,c)\otimes q^2(x^2;b^2,c)$, as desired.
\end{proof}
\begin{prop}\label{P:cyc_2}
  If $T$ is totally transcendental, then it is \icyc if and only if
  there is no proper cycle of types in $T^\mathrm{eq}$ of length $2$.
\end{prop}
\begin{proof}
  By Corollary \ref{C:teq}, we need only prove one direction, so
  suppose \[ p_0(x,y)\pf\dotsb\pf p_{n-1}(x,y)\pf p_0(x,y)\] is a
  proper cycle of types with real variables. Since $T$ is totally
  transcendental, there is a finite tuple $c$ in $\aclq(\emptyset)$
  such that all nonforking extensions of all $p_i$'s to $c$ are
  stationary. Lemma \ref{L:erweiterung} gives a proper cycle of
  stationary types
  \[ p_0(x;y,c)\pf\dotsb\pf p_{k-1}(x;y,c)\pf p_0(x;y,c)\] for some $k$ in $\N$.

  Denote by $\bar x=(x^0,\ldots, x^{k-2})$ and $\bar y=(y^0,\ldots,
  y^{k-2})$ and consider the types
  \[ \begin{aligned}
    r_1( \bar x; \bar y ,c) &= p_0(x^0,y^0, c)\otimes p_1(x^1,y^1, c)
    \otimes \ldots \otimes p_{k-2}(x^{k-2}, y^{k-2}, c)\\
    r_2(\bar x; \bar y,c) &= p_1(x^0,y^0, c)\otimes p_2(x^1,y^1, c)
    \otimes \ldots \otimes p_{k-1}(x^{k-2}, y^{k-2}, c)\\
    r_3(\bar x; \bar y,c) &= p_1(x^0,y^0, c)\otimes p_2(x^1,y^1, c)
    \otimes \ldots \otimes p_{k-2}(x^{k-3}, y^{k-3}, c) \otimes
    p_0(x^{k-2}, y^{k-2}, c)\\
 \end{aligned}\]

  \noindent The Lemma \ref{L:pf_tensor} yields the cycle of types \[
  r_1(\bar x; \bar y,c)\pf r_2(\bar x; \bar y,c)\pf r_3(\bar x; \bar
  y,c).\]

  Given $(\bar a, \bar b)$ realising $r_1(\bar x; \bar y ,c)$ and
  $(\bar a', \bar b')$ realising $r_2(\bar x; \bar y ,c)$, notice
  that \[ (a^1, a^2, \ldots, a^{k-2}, a^0, b^1, b^2, \ldots, b^{k-2},
  b^0)\] realise $r_3(\bar x;\bar y, c)$. If $f$ denotes the function
  which maps a $k-1$-tuple $(f^0, \ldots, f^{k-1})$ to the imaginary
  coding the set $\{f^1,\ldots, f^{k-1}\}$, Corollary \ref{C:pf_f}
  implies that
  \begin{multline*} \tp(\{ a^0, \ldots, a^{k-2}\};
    \{ b^0, \ldots, b^{k-2}\}, c )\pf \tp(\{ a^1, \ldots, a^{k-1}\};
    \{ b^1, \ldots, b^{k-1}\}, c ) \pf \\ \pf \tp( \{ a^0, \ldots,
    a^{k-2} \}; \{ b^0, \ldots, b^{k-2}\}, c ) \end{multline*}
  \noindent In order to conclude, we need only show that the above two
  imaginary types are different. Otherwise, if the two types are
  equal, we have for each $1\leq i\leq k-1$, two values $0\leq
  \rho(i), \tau(i)\leq k-2$ such that $(a^{\rho(i)},
  b^{\tau(i)})\models p_i(x, y, c)$. Observe that no two elements
  $a^i$ and $a^j$, with $i\neq j$,  can be equal since the independence
  $a^i\ind_c a^j$ would imply that $a_i$ is algebraic, and thus
  $p_{i+1}(x,y, c)=p_i(x,y,c)$, by the Remark \ref{R:pf_alg}.
  Likewise, no two elements $b^i$ and $b^j$ can be equal, for $i\neq
  j$. Thus, each of the maps $i\mapsto \rho(i)$ and $i \mapsto
  \tau(i)$ is injective, so both $\rho$ and $\tau$ are bijections.

  If $\rho(k-1)=\tau(k-1)=j$, then $(a^j, b^j)$ realises both
  $p_j(x;y,c)$ and $p_{k-1}(x;y,c)$, which contradicts that the cycle
  of types is proper. Hence, the values $\rho(k-1)$ and $\tau(k-1)$
  are different, so there must be some $1\leq i\leq k-2$ such that
  $\rho(i)\neq \tau(i)$. The independences \[ a^{\rho(i)} \ind_c
  b^{\tau(i)} \text { and } a^{\rho(k-1)} \ind_c b^{\tau(k-1)}\]
  imply
  that $p_i(x,y, c)=p_{k-1}(x,y,c)$, by the Remark \ref{R:pf_alg} and
  stationarity of strong types, which yields the desired
  contradiction.

\end{proof}

We do not know whether Corollary \ref{C:teq} and Proposition
\ref{P:cyc_2} are true for arbitrary stable theories.

All known examples of non-equational stable theories have a proper
cycle of real types of length $2$. Indeed, in Hrushovski and Srour's
primordial example \cite{HS89}, the type of a white point and the type
of a red point in a plane indiscernibly converge to each other,
whereas the non-abelian free group satisfies the MS-criterion
\cite[Lemmata 3.4 \& 3.6]{MS17}. In this note, we will provide new
examples of non-equational totally transcendental theories, one for
each natural number $k$, having proper cycles of length $k$ but no
proper cycles of real types of length strictly smaller than $k$. We
will do so by suitable colouring the free pseudospace, mimicking the
construction of Hrushovski and Srour. The following question seems
hence natural, though we do not have a solid guess what the answer
will be.

\begin{question}
  Is there a non-equational \icyc theory?
\end{question}

Related to the above, we wonder whether there is a local
characterisation of equationality in terms of cycles of types:

\begin{question}
  Is a formula $\varphi(x,y)$ a Boolean combination of equations if
  and only if whenever \[ \varphi \in p_0(x,y)\pf p_1(x,y) \pf\ldots
  \pf p_{n-1} (x,y)\pf p_0(x,y),\] then $\varphi$ belongs to $p_i$ for
  every $i>0$?

  Do two types $ p$ and $q$ contain the exact same equations if and
  only if $p$ and $q$ both occur in a (proper) cycle of types?
\end{question}

Observe that a positive answer to the second question would positively
answer the first one.
\section{Indiscernible Kernels}\label{S:Kernel}

To our knowledge, the results in this section only appeared in print
form in Adler's Master's Thesis \cite{hA96} (in German). Therefore, we
will include their proofs, even if the results are most likely
well-known among the community.

As before, work inside a sufficiently saturated model $\mathbb U$ of
the complete theory $T$.
\begin{notation}
  Given two subsets $I_0$ and $I_1$ of a linearly ordered infinite
  index set with no endpoints, we write $I_0\ll I_1$ if $i_0<i_1$ for
  all $i_0$ in $I_0$ and $i_1$ in $I_1$. If $(a_i)_{i\in I}$ is a
  sequence indexed by $I$, set $\aclq(a_{I_0})=\acl(\{a_i\}_{i\in
    I_0})$.
\end{notation}

\begin{definition}\label{D:Kern}
  The \emph{kernel} of the indiscernible sequence $(a_i)_{i\in I}$ is
  defined as
  \[ \Ker((a_i)_{i\in I}) =\bigcup\limits_{\stackrel{I_0, I_1
    \subset I}{I_0\ll I_1}} \aclq(a_{I_0})\cap \aclq(a_{I_1}). \]
\end{definition}

\noindent Note that we may assume that both $I_0$ and $I_1$ are finite
subsets of $I$. Furthermore, the set $\aclq(a_{I_0})\cap
\aclq(a_{I_1})$ only depends on $|I_0|$ and $|I_1|$ (possibly after
enlarging $I$), since $(a_i)_{i\in J}$ is indiscernible
over $a_{I_0}$, whenever the segment $J\subset I\setminus I_0$
is either left of
$I_0$ or right of $I_0$. If the sequence is indiscernible as a set
(which is
always the case in stable theories), then we may define the kernel by
considering all the intersections given by pairs $(I_0, I_1)$ with
$I_0\cap I_1=\emptyset$.

\noindent Observe that (if $I$ is large enough),
\[ \Ker((a_i)_{i\in I}) = \aclq(a_{I_0})\cap \aclq(a_{I_1}),
\text{ for any } I_0\ll I_1 \text{ both infinite}.\]

\begin{lemma}\label{L:Ker_large}
  The kernel $K$ of an indiscernible sequence $(a_i)_{i\in I}$ is the
  largest subset of $\aclq((a_i)_{i\in I} )$ over which the sequence
  is indiscernible.
\end{lemma}

\begin{proof}
  We may assume that $I$ has no endpoints. Clearly, the sequence is
  indiscernible over $K$. Given a tuple $b$ in $\aclq(a_{I_0})$, for
  $I_0\subset I$ finite, such that the sequence is indiscernible over
  $b$, the tuple $b$ also lies in $\aclq(a_{I_1})$, whenever the
  segment $I_1\gg I_0$ has the same size as $I_0$,  so
  $b$ lies in the kernel \[K=\bigcup\limits_{\stackrel{J_0, J_1
  		\subset I}{J_0\ll J_1}} \aclq(a_{J_0})\cap \aclq(a_{J_1}).\]
\end{proof}

\begin{lemma}\label{L:Kern_small}
  If $T$ is stable, then the kernel $K$ of an indiscernible sequence
  $(a_i)_{i\in I}$ is the smallest algebraically closed subset (in
  $T^\mathrm{eq}$) over which the sequence is independent.
\end{lemma}

\begin{proof}
  Let $E$ be an algebraically closed subset (in $T^\mathrm{eq}$) such
  that $(a_i)_{i\in I}$ is $E$-independent. In particular, for each
  $I_0<I_1$, we have that
  \[ a_{I_0}\ind_Ea_{I_1},\]
  \noindent so $K\subset E$.

  \noindent Let now $\mathfrak{p}=\mathrm{Av}((a_i)_{i\in I})$ be the
  average type, that is,
  \[ \mathfrak p=\{ \varphi(x, b)\ \LL_{\mathbb U}\text{-formula}\ \mid \
  \varphi(a_i, b) \text{ for all but finitely many } i \in I \}.\]

  \noindent Since $\mathfrak p$ is invariant over every infinite
  subsequence of $(a_i)_{i\in I}$, its canonical base $C$ is contained
  in $K$. Thus, the sequence is $C$-indiscernible and $\mathfrak p$ is
  a nonforking extension of the stationary type $\mathfrak p\restr K$.

  \noindent It suffices to show that $a_i\models \mathfrak
  p\restr{K\cup (a_j)_{j<i}}$, since any Morley sequence of $\mathfrak
  p\restr K$ has this property and its type over $K$ is unique. Thus,
  let $\varphi(x, (a_j)_{j<i})$ be a formula in $\mathfrak
  p\restr{K\cup (a_j)_{j<i}}$. We may clearly assume that $I$ has no
  last element. By definition of the average type, there is some
  $a_t\models \mathfrak p\restr{K\cup (a_j)_{j<i}}$ with $t\geq i$. By
  indiscernibility,
  \[a_i\models \mathfrak p\restr{K\cup (a_j)_{j<i}},\]  as desired.
 \end{proof}

\begin{cor}\label{C:Kern_MF}
  In a stable theory $T$, every indiscernible sequence is a Morley
  sequence over its kernel.
\end{cor}

Using kernels, we can provide a different characterisation of the
relation $\pf$ in a stable theory.

\begin{cor}\label{C:Kern_Pf}
  Given types $p(x;y)$ and $q(x;y)$ in a stable theory $T$, we have
  that $p\pf q$ if and only if there is a set $C$ and tuples $a$, $a'$
  and $b$ such that:
  \begin{itemize}
  \item $\models p(a,b)$ and $\models q(a',b)$;
  \item $a\stackrel{\mathrm stp}{\equiv}_C a'$, and
  \item $a\ind_C b$.
  \end{itemize}
  \noindent In particular, given a cycle
  \[ p_0(x,y)\pf p_1(x,y)\pf \ldots \pf
  p_{n-1}(x,y)\pf p_0(x,y),\] there are tuples $b$, $a_0, \ldots, a_n$
  and subsets $C_0,\ldots, C_{n-1}$ such that:
  \begin{itemize}
  \item $\models p_r(a_r,b)$, for $0\leq r\leq n-1$, and $\models
    p_0(a_n,b)$.
  \item $a_r\stackrel{\mathrm stp}{\equiv}_{C_r} a_{r+1}$ for $0\leq
    r\leq n-1$.
  \item $a_r\ind_{C_r} b$ for $0\leq r\leq n-1$.
  \end{itemize}
\end{cor}

\begin{proof}
  If $p\pf q$, choose some tuple $b$ and an indiscernible sequence
  $(a_i)_{i<|T|^+}$ such that $q(a_0, b)$ and $p(a_i, b)$ for each
  $i>0$. Consider the kernel $K$ of the sequence, which is
  algebraically closed in $T^\mathrm{eq}$, so $a_i \stackrel{\mathrm
    stp}{\equiv}_K a_j$ for all $i, j$. In particular, the subsequence
  $(a_i)_{0<i<|T|^+}$ is Morley sequence over $K$, so there is some
  $i_0<|T|^+$ such that $a_{i_0}\ind_K b$. Set $C=K$, $a'=a_0$ and
  $a=a_{i_0}$.

  \noindent For the other direction, set $a_0=a'$ and choose for each
  $0\neq i$ in $\mathbb N$ a realisation $a_i\stackrel{\mathrm
    stp}{\equiv}_C a$ such that $a_i\ind_C b, \{a_j\}_{j<i}$. Since
  strong types are stationary, we have that $p(a_i, b)$, for $i\neq
  0$. Furthermore, the sequence $\{a_i\}_{i\geq 0}$ is indiscernible
  over $C$, by construction.
\end{proof}

The above provides a simpler characterisation of equations in stable
theories (cf. \cite[Remark 2.4]{eH12}).

\begin{remark}\label{R:Hrushovski}
  In a stable theory $T$, a formula $\varphi(x;y)$ is an equation if
  and only if for every set set $C$ and tuples $a$, $a'$ and $b$ such
  that $\varphi(a,b)$ holds with $a\ind_C b$, then so does
  $\varphi(a',b)$ hold, whenever $a'\stackrel{\mathrm stp}{\equiv}_C
  a$.
\end{remark}

\begin{proof}
  Given $C$, $a$, $a'$ and $b$ as in the statement, Corollary
  \ref{C:Kern_Pf} yields that $\tp(a,b)\pf \tp(a', b)$. As $\varphi$
  belongs to $\tp(a,b)$, it must lie in $\tp(a',b)$, by Corollary
  \ref{C:pf_cl}.

  For the other direction, it suffices to show that $\varphi$ lies in
  $q$, whenever $\varphi$ belongs to $p$ and $p\pf q$, by Corollary
  \ref{C:pf_cl}. By Corollary \ref{C:Kern_Pf}, there are $C$, $a$,
  $a'$ and $b$ such that $p(a,b)$, $q(a',b)$, $a\ind_C b$ and
  $a'\stackrel{\mathrm stp}{\equiv}_C a$. Since $\varphi(a,b)$ holds,
  we conclude that so does $\varphi(a',b)$, that is, the formula
  $\varphi$ belongs to $q$, as desired.
\end{proof}
\section{A blank pseudospace}\label{S:PS}

Hrushovski and Srour produced the first example \cite{HS89} of a
non-equational stable theory by adding two colours to an underlying
($2$-dimensional) free pseudospace, a structure later studied by
Baudisch and Pillay \cite{BP00}. Subsequently, the free
($n$-dimensional) pseudospace has been considered from different
perspectives, either as a lattice \cite{kT14, TZ17} or as a
right-angled building \cite{BMPZ14, BMPZ17}, in order to show that the
ample hierarchy is strict. In this section, we will recall the basic
properties of the free $2$-dimensional pseudospace.

A \emph{geometry} is a graph whose vertices have levels $0$, $1$ and
$2$. Vertices of level $0$ are called \emph{points} (usually denoted
by the letter $c$), whereas vertices of level $1$ are \emph{lines}
(denoted by $b$) and vertices of level $2$ are \emph{planes} (denoted
by $a$). In an abuse of notation, we say that \emph{the point} $c$
\emph{lies in the plane} $a$ if there is a line $b$ contained in $a$
passing through $c$, though there are no edges between points and
planes. We refer to a subgraph of the form $a-b-c$ as a \emph{flag}.

A \emph{letter} $s$ is a non-empty subinterval of $[0,2]$. Given a
flag $F$ in a geometry $A$, a new geometry $B$ is obtained from $A, F$
via the \emph{operation} $s$ by freely adding a new flag $G$ which
agrees with $F$ on the levels in $[0,2]\setminus s$:

\begin{figure}[!htbp]
\begin{subfigure}[h]{0.20\textwidth}

  \resizebox{\linewidth}{!}{
    \begin{tikzpicture}

      \fill  (0,2)  node [left]  {$a$}  circle (2pt);
      \fill   (0,1)  node [left]  {$b$} circle (2pt);
      \fill   (0,0)  node [left] {$c$} circle (2pt);
      \fill (-1/3, -1/3)  node [left]  {$F$} ;
      \fill   (1,0)  node [right] {$c'$} circle (2pt);
      \fill (4/3, -1/3)  node [right]  {$G$} ;

      \draw  (0,2) --  (0,1) --  (0,0) ;
      \draw (0,1) --  (1,0) ;
      \draw[->]  (0, -1/3) to node [pos=.5, below=1mm] {$[0]$} (1, -1/3) ;

		\end{tikzpicture}
	}\vskip-3mm
	\caption*{Operation $[0]$}
\end{subfigure}	\hspace{1cm}
\begin{subfigure}[h]{0.20\textwidth}

  \resizebox{\linewidth}{!}{
    \begin{tikzpicture}

      \fill  (0,2)  node [left]  {$a$}  circle (2pt);
      \fill   (0,1)  node [left]  {$b$} circle (2pt);
      \fill   (0,0)  node [left] {$c$} circle (2pt);
      \fill (-1/3, -1/3)  node [left]  {$F$} ;
      \fill   (1,1)  node [right] {$b'$} circle (2pt);
      \fill (4/3, -1/3)  node [right]  {$G$} ;

      \draw  (0,2) --  (0,1) --  (0,0) ;
      \draw (0,2) --  (1,1) -- (0,0) ;
      \draw[->]  (0, -1/3) to node [pos=.5, below=1mm] {$[1]$} (1, -1/3) ;

    \end{tikzpicture}
  }
  \vskip-3mm
  \caption*{Operation $[1]$}
\end{subfigure}	\hspace{1cm}
\begin{subfigure}[h]{0.20\textwidth}

  \resizebox{\linewidth}{!}{
    \begin{tikzpicture}

      \fill  (0,2)  node [left]  {$a$}  circle (2pt);
      \fill   (0,1)  node [left]  {$b$} circle (2pt);
      \fill   (0,0)  node [left] {$c$} circle (2pt);
      \fill (-1/3, -1/3)  node [left]  {$F$} ;
      \fill   (1,2)  node [right] {$a'$} circle (2pt);
      \fill (4/3, -1/3)  node [right]  {$G$} ;

      \draw  (0,2) --  (0,1) --  (0,0) ;
      \draw (0,0) --  (0,1) --  (1,2) ;
      \draw[->]  (0, -1/3) to node [pos=.5, below=1mm] {$[2]$} (1, -1/3) ;
    \end{tikzpicture}
}
  \vskip-3mm
  \caption*{Operation $[2]$}
\end{subfigure}
\vskip8mm
\begin{subfigure}[h]{0.25\textwidth}
  \resizebox{\linewidth}{!}{
    \begin{tikzpicture}[->=latex,text height=1ex,text depth=1ex]

      \fill  (0,0)  node [below left]  {$c$}  circle (2pt);
      \fill   (-1/2,1)  node [below left]  {$b$} circle (2pt);
      \fill   (-1,2)  node [below left] {$a$} circle (2pt);
      \fill (-4/3, 7/3)  node [left]  {$F$} ;
      \fill   (1/2,1)  node [below right] {$b'$} circle (2pt);
      \fill   (1,2)  node [below right] {$a'$} circle (2pt);
      \fill (4/3, 7/3)  node [right]  {$G$} ;

      \draw  (0,0) --  (-1/2,1) --  (-1,2) ;
      \draw (0,0) --  (1/2,1) --  (1,2) ;
      \draw[bend left=30,->]  (-4/3, 7/3) to node [pos=.5, below=2mm] {$[1,2]$} (4/3,
      7/3);
	\end{tikzpicture}
  }\vskip-3mm
  \caption*{Operation $[1,2]$}
\end{subfigure}\hspace{1cm}
\begin{subfigure}[h]{0.25\textwidth}

	\resizebox{\linewidth}{!}{
	  \begin{tikzpicture}[->=latex,text height=1ex,text depth=1ex]

	    \fill  (0,2)  node [left]  {$a$}  circle (2pt);
	    \fill   (-1/2,1)  node [left]  {$b$} circle (2pt);
	    \fill   (-1,0)  node [left] {$c$} circle (2pt);
	    \fill (-4/3, -1/3)  node [left]  {$F$} ;
	    \fill   (1/2,1)  node [right] {$b'$} circle (2pt);
	    \fill   (1,0)  node [right] {$c'$} circle (2pt);
	    \fill (4/3, -1/3)  node [right]  {$G$} ;

	    \draw  (0,2) --  (-1/2,1) --  (-1,0) ;
	    \draw (0,2) --  (1/2,1) --  (1,0) ;
	\draw[bend right=30,->]  (-4/3, -1/3) to node [pos=.5, below right=2mm]
             {$[0,1]$} (4/3, -1/3);
	  \end{tikzpicture}
        }\vskip-3mm
        \caption*{Operation $[0,1]$}
\end{subfigure}
\begin{subfigure}[h]{0.22\textwidth}

  \resizebox{\linewidth}{!}{
    \begin{tikzpicture}[->=latex,text height=1ex,text depth=1ex]

      \fill  (0,2)  node [left]  {$a$}  circle (2pt);
      \fill   (0,1)  node [left]  {$b$} circle (2pt);
      \fill   (0,0)  node [left] {$c$} circle (2pt);
      \fill (-1/3, -1/3)  node [left]  {$F$} ;
      \fill  (1,2)  node [right]  {$a'$}  circle (2pt);
      \fill   (1,1)  node [right] {$b'$} circle (2pt);
      \fill   (1,0)  node [right] {$c'$} circle (2pt);
      \fill (4/3, -1/3)  node [right]  {$G$} ;

      \draw  (0,2) --  (0,1) --  (0,0) ;
	\draw (1,2) --  (1,1) --  (1,0) ;
	\draw[->]  (0, -1/3) to node [pos=.5, below=1mm] {$[0,2]$} (1, -1/3);

	\end{tikzpicture}
  }\vskip-3mm
  \caption*{Operation $[0,2]$}
\end{subfigure}
\end{figure}

The \emph{free pseudospace} $\mi$ is obtained by successively applying
countably many times all of the above operations starting from a flag.
The geometry $\mi$ is independent, up to isomorphism, of the order in
which the operations are applied. It is denoted by $\mi^2$ in
\cite[Definition 4.6]{BMPZ14}. Observe that the geometry obtained
by
only considering the operations $0$, $1$ and $2$ is an elementary
substructure (\cf \cite[Corollary 4.31]{BMPZ14}) of $\mi$
(namely, the prime model).

\bigskip

We will now exhibit the axioms for the theory $\ps$ of $\mi$. Let us
first fix some notation. We will write the symbol $i$ for the letter
$[i]$, that is, from now on, we will write $0$ instead of $[0]$,
etc.  A \emph{word} is a sequence of letters. A
\emph{permutation} of the word $u$ is obtained by successively
replacing an occurrence of the subword $0\cdot 2$ by the subword
$2\cdot 0$; similarly the subword $2\cdot 0$ is permuted to $0\cdot
2$ (Note that $0\cdot 2$ denotes the word $[0]\cdot [2]$, but
we do not want to render the notation cumbersome with additional
brackets). The word $u$ is \emph{reduced} if it does not contain,
up to
permutation, a subword of the form $s\cdot t$, where $s\subset t$ or
$t\subset s$ (please note that our notation $s\subset t$ does not
imply $s\subsetneq t$).

A \emph{flag path} \[F_0\lmto{s_1}F_1\dotsb F_{n-1}\lmto{s_n}F_n\]
with word $u=s_1\dotsb s_n$ is a sequence of flags such that, for each
$1\leq i\leq n$, the flag $F_i$ agrees with $F_{i-1}$ exactly in the
levels in $[0,2]\setminus s_i$. The above flag path is \emph{reduced}
if its word is reduced and for each $i$, the flags $F_{i-1}$ and $F_i$
cannot be connected by a \emph{splitting}, that is, a flag subpath
whose word consists of proper subletters of $s_i$. It is not hard to
show that every two flags are connected by a reduced path
\cite[Corollary 3.13]{BMPZ17}.

\begin{fact}\label{F:PS}\cite[Theorem 4.12]{BMPZ17}
  The theory $\ps$ is axiomatised by the following properties:
  \begin{enumerate}
  \item The universe is a geometry such that every vertex lies in a flag.
  \item For every level $i$ in $[0,2]$ and every flag $F$, there are
    infinitely many flags $G$ with $F\lmto{i}G$.
  \item\label{F:PS:simple} Every closed reduced flag path
    $F_0\lmto{s_1}F_1\dotsb F_{n-1}\lmto{s_n}F_0$ has length $n=0$.
  \end{enumerate}
\end{fact}
\noindent It was proven in \cite[Theorem 3.26]{BMPZ17} that property
(\ref{F:PS:simple}) can be expressed by a set of elementary sentences.

We will now describe types and the geometry of forking in the
pseudospace. We refer the reader to \cite[Sections 3--7]{BMPZ17} for
the corresponding proofs. Since there are no non-trivial reduced
closed paths of flags, the word $u$ connecting two flags $F$ and $G$
by a reduced path $F\lmto u G$ is unique, up to permutation, and will
be denoted by $\dd(F,G)$. The flags $F$ and $G$ agree modulo a subset
$S$ of $[0,2]$, that is, they have the same vertices in all levels off
$S$, if and only if the letters in $\dd(F,G)$ are all contained in
$S$. In particular, the collection of points and lines, resp. lines
and planes, form a pseudoplane, so every two lines intersect in at
most one point, resp. lie in at most one plane. Furthermore, the
intersection of two distinct planes is either empty, a unique point or
a unique line \cite[Axiom $\Sigma3$]{BP00}. Actually, the
geometry forms a lattice,
once a smallest element $\mathbf 0$ and a largest element $\mathbf 1$
are added \cite{kT14, TZ17}.

If $u=\dd(F,G)=u_1\cdot u_2$, given two reduced flag paths
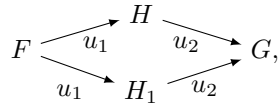
\begin{figure}[h]
  \centering

  \begin{tikzpicture}[>=latex,->]

    \matrix (A) [matrix of math nodes,column sep=1cm]
	    {  & H &  \\
	      F &   &   G, \\
	      &  H_1 &    \\};

	    \draw (A-2-1)  edge node[pos=.7, below] {$u_1$} 	(A-1-2) ;
	    \draw (A-2-1)  edge node[pos=.7, below left] {$u_1$} 	(A-3-2) ;

	    \draw (A-1-2)  edge node[pos=.3, below] {$u_2$} 	(A-2-3) ;
	    \draw (A-3-2)  edge node[pos=.5, below] {$u_2$} 	(A-2-3) ;

  \end{tikzpicture}

\end{figure}

\noindent and a vertex $p$ in $H$ of level $i$ which does not
\emph{wobble}, that is, such that the word $u_1\cdot i$ or the word $ i\cdot
u_2$ is
reduced, then $p$ is also a vertex of $H_1$. In particular, the vertex
$p$ is definable over $F, G$.

A non-empty subset $A$ of $\mi$ is \emph{nice} if:
\begin{itemize}
\item every vertex in $A$ lies in a flag fully contained in $A$; and
\item every two flags in $A$ are connected by a reduced path of flags
  in $A$.
\end{itemize}
\noindent
Algebraic closure and the definable closure of a set $X$ agree
\cite[Proposition 2.29]{kT14} and coincide with the intersection of
all nice sets $A\supset X$ \cite[Corollary 2.27]{kT14} (see also
\cite[Proposition 5.1\ \&  Corollary 5.4]{TZ17}). If $X$ is finite,
then so is the algebraic closure. The quantifier-free type of a nice
subset determines its type. More generally:

\begin{fact}\label{F:qf_alg}\textup{(}\cf\
	\cite[Remark 2.26]{kT14} \ \& \cite[Corollary
	3.12]{TZ17}\textup{)}
  The quantifier-free type of an algebraically closed subset
  determines its type in $\ps$.
\end{fact}

Observe that if we apply one of the operations $[0]$, $[1]$ or $[2]$
to a flag in a nice set $A$, the resulting geometry is again nice.

Given a flag $F$ and a nice subset $A$, there is a flag $G$ in $A$
(called a \emph{base-point} of $F$ over $A$) such that, for any flag
$G'$ in $A$, the word $\dd(F,G')$ is the \emph{non-splitting
  reduction} of $\dd(F,G)\cdot \dd(G, G')$, that is, whenever a subword
$s\cdot t$ or $t\cdot s$ occurs in a permutation of the product
$\dd(F,G)\cdot \dd(G, G')$, with $s\subset t$, we cancel $s$. If we
consider a reduced flag path $P$ connecting $F$ to some base-point $G$
over $A$ with word $\dd(F,G)$, the set $A\cup P$ is again nice. Any
flag occurring in the nice set $P$ appears in a permutation of the
path $P$.

The theory $\ps$ of $\mi$ is $\omega$-stable of rank $\omega^2$,
equational with perfectly trivial forking and has weak elimination of
imaginaries. Forking can be easily described: Given nice sets $A$ and
$B$ containing a common algebraically closed subset $C$, we have that
$A\ind_C B$ if and only if for every nice set $D\supset C$ and flags
$F$ in $A$ and $H$ in $B$ we have that $\dd(F, H)$ is the
non-splitting reduction of $\dd(F,G)\cdot \dd(G, H)$, where $G$ is a
base-point of $F$ over $D$. In particular,

\[F\ind_G D.\]

\begin{remark}\label{R:ver_acl}\textup{(}\cite[Theorem
2.35]{kT14} \ \& \cite[Proposition 4.3 \& Theorem
4.13]{TZ17}\textup{)}
  Assume that $A$, $B$ and $C=A\cap B$ are algebraically closed and
  $A\ind_C B$. Then
  \begin{enumerate}
  \item\label{R:ver_acl:acl}$A\cup B$ is algebraically closed,
  \item\label{R:ver_acl:direct}if a vertex $x$ in $A$ is directly
    connected to a vertex $y$ in $B$, then $x$ or $y$ must lie in $C$,
  \item\label{R:ver_acl:pp}if a point in $A$ lies in a plane of $B$,
    then there is a line in $C$ connecting them,
  \item\label{R:ver_acl:inf} a point $c$, which belongs to both a line in
    $A\setminus C$ and to a line in $B\setminus C$, lies in $C$.
  \end{enumerate}
\end{remark}

Before introducing the $k$-colored pseudospace in section
\ref{S:ColPS}, we will prove several auxiliary results about the free
pseudospace. We hope that this will allow the reader to become more
familiar with the theory $\ps$.

\begin{lemma}\label{L:AP_unabh}
  Let $X$ and $Y$ be algebraically closed sets independent over their
  common intersection $Z$. Given a point $c$ not contained in
  $Y\setminus Z$ lying in the line $b$ of $X$, then \[
  X\cup\{c\}\ind_Z Y.\]
\end{lemma}

\begin{proof}
  By the transitivity of non-forking, we may assume that $Z=X$. If $c$
  belongs to $X$, then there is nothing to prove. Otherwise, the type
  of $c$ over $X$ has Morley rank $1$ (it is actually strongly
  minimal), by \cite[Remark 6.2]{BP00} (cf. \cite[Corollary
    7.13]{BMPZ14}). Since the extension $\tp(c/Y)$ is not algebraic,
  it does not fork over $X$.
\end{proof}

\begin{lemma}\label{L:pairwise}
  The type of a set $X$ is determined by the collection of types
  $\tp(x,x')$, with $x$ and $x'$ in $X$.

  In particular, if $X\equiv_Z X'$ and $X\equiv_Y X'$, then
  $X\equiv_{YZ} X'$.
\end{lemma}

\begin{proof}
  Choose an enumeration of $X=\{x_\alpha\}_{\alpha <\kappa}$ and flags
  $F_\alpha$ containing $x_\alpha$, for $\alpha<\kappa$, such that
  $F_\alpha\ind_{x_\alpha} X \cup\{F_\beta\}_{\beta<\alpha}$. In
  particular, for $\alpha\neq \beta$, we have that \[
  F_\alpha\ind_{x_\alpha} x_\beta \text{ \quad and \quad }
  F_\alpha\ind_{x_\alpha, x_\beta} F_\beta. \]

  \noindent Since the type of $F_\alpha$ over $x_\alpha$ is
  stationary, the type of the pair $(x_\alpha, x_\beta)$ determines
  the type of $F_\alpha, F_\beta$. By \cite[Theorem 7.24]{BMPZ17}, the
  type of $(F_\alpha)_{\alpha<\kappa}$, hence the type of $X$, is
  uniquely determined by the collection of types
  $\tp(x_\alpha,x_\beta)$, for $\alpha, \beta<\kappa$.
\end{proof}
\section{A colored pseudospace}\label{S:ColPS}

Work inside a sufficiently saturated model $\mathbb U$ of the theory
$\ps$ of the free pseudospace and consider a natural number $k\geq 2$.
For $0\leq i<k$, we use the notation $i+1$ instead of $i+1 \textrm{
  mod } k$, and likewise $i-1$ for $i-1 \textrm{ mod } k$.

We colour the lines in $\mathbb U$, as well as the pairs $(a,c)$,
where the point $c$ lies in the plane $a$, with $k$ many colours.
Formally, we partition the set of lines into subsets $C_0, \ldots,
C_{k-1}$, and the set of pairs $(a,c)$, where $c$ lies in the plane
$a$, into $I_0,\ldots, I_k$. Given a plane $a$ and an index $0\leq
i<k$, we denote by the \emph{section} $I_i(a)$ the collection of
points $c$ with $I_i(a,c)$.

\bigskip
Consider the theory $\cps_k$ of $k$-colored pseudospaces with
following axioms:

\begin{itemize}
\item The axioms of $\ps$.
\end{itemize}
\noindent {\sc Universal Axioms}
\begin{itemize}
\item For each $0\leq i <k$, given a line $b$ with colour $i$ in a
  plane $a$, all the points $c$ in $b$ lie in the section $I_i(a)$
  except at most one point, which lies in $I_{i+1}(a)$ (if such a point exists,
  we call it the \emph{exceptional} point of $b$ in $a$).
\end{itemize}
\noindent {\sc Inductive Axioms}
\begin{itemize}
\item Every line $b$ in a plane $a$ contains an \emph{exceptional}
  point, denoted by $\expt(a,b)$.
\item For each $0\leq i <k$, given a point $c$ and a plane $a$ with
  $I_i(a,c)$, there are infinitely many lines in $a$ passing through
  $c$ with colour $i$.
\item For each $0\leq i <k$, given a point $c$ and a plane $a$ with
  $I_i(a,c)$, there are infinitely many lines in $a$ passing through
  $c$ with colour $i-1$.
\item For every point $c$ in a line $b$, there are infinitely many
  planes $a$ containing $b$ such that $c$ is exceptional for $b$ in
  $a$.
\end{itemize}

\bigskip

We can construct a model of $\cps_k$ as follows: We start with a flag
$A_0=\{a-b-c\}$ with any colouring, eg.\ $b\in C_0$ and $I_0(a,c)$ and
construct an ascending sequence $A_0\subset A_1\subset\dotsb$ of
coloured geometries by applying one the operations $[0]$, $[1]$ and
$[2]$ to a flag $a-b-c$ in $A_j$ obtain $A_{j+1}$, extending the
colouring to $A_{j+1}$ in an arbitrary way whilst preserving the
Universal Axioms. For example, do as follows:

\begin{itemize}
\item Operation [0] adds a new point $c'$ to $b$. If $b$ has colour
  $i$, then for all $a''$ in $A_j$ containing $b$, paint the pair
  $(a'',c')$ with the colour $i$, if $\expt(a'',b)$ already exists in
  $A_j$. Otherwise, paint $(a'',c')$ with the colour $i+1$.
\item Operation [1] adds a new line $b'$ between $a$ and $c$. If
  $(a,c)$ has colour $i$, then paint $b'$ with the colour $i$ or the
  colour $i-1$, and see to it that each choice occurs infinitely often
  in the sequence.
\item Operation [2] adds a new plane $a'$ which contains $b$. If $b$
  has colour $i$, then for all $c''$ in $A_j$ which lie in $b$, we
  give the pair $(a',c'')$ one of the colours $i$ or $i-1$. Each
  choice should occur infinitely often.
\end{itemize}

\noindent It is easy to see that the structure obtained in this
fashion satisfies all axioms of $\cps_k$, so the theory $\cps_k$ is
consistent.

\bigskip

\begin{notation}
  Given a subset $X$ of a model of $\cps_k$, we will denote by $\fcl
  X$ the algebraic closure of $X$ in the reduct $\ps$, and by
  $\EP(X)=\{\expt(a,b)\,,\, (a, b) \in X\times X\}$ the exceptional
  points of lines and planes from $X$.
\end{notation}

\begin{remark}\label{R:acl_color}
  If the point $c$ is directly connected to a line in $X$, then
  $\fcl{X,c}=\fcl{X}\cup \{c\}$.

  In particular, if $X=\fcl X$, given $c$ in $\EP(X)$, then $X\cup
  \{c\}$ is algebraically closed in the reduct $\ps$.
\end{remark}

\begin{proof}
  In order to show that $\fcl{X,c}=\fcl{X}\cup \{c\}$, it suffices to
  consider the case when $X$ is nice. The geometry $X\cup\{c\}$ is
  either $X$ or obtained from $X$ by applying the operation $[0]$, so
  it is nice again, and thus algebraically closed.
\end{proof}

Similar to \cite[Proposition 4.26]{BMPZ14} or \cite[Proposition
3.10]{TZ17}, it is easy to see that the
collection of partial isomorphisms between $\ps$-algebraically closed
finite sets closed under exceptional points  inside two
$\aleph_0$-saturated models of $\cps_k$ is non-empty and
has the back-and-forth property, so we deduce the following:

\begin{theorem}\label{T:CPSk_vollst}
  The theory $\cps_k$ is complete. Given a set $X$ in a model of
  $\cps_k$ with $X=\fcl X$ and $\EP(X)\subset X$, then the
  quantifier-free type of $X$ determines its type.
\end{theorem}

The back-and-forth system yields an explicit description of the
algebraic closure, as well as showing that the theory $\cps_k$ is
$\omega$-stable, by a standard type-counting argument.

\begin{cor}\label{C:col_stable}
  The theory $\cps_k$ is $\omega$-stable. The algebraic closure
  $\acl(X)$ of a set $X$ is obtained by closing $\fcl X$ under
  exceptional points:
\[\acl(X)= \fcl X\cup \EP(\fcl X).\]
\end{cor}

We deduce the following characterisation of forking over (colored)
algebraically closed sets.
 \begin{cor}\label{C:color_Forking}
   Let $X$ and $Y$ two supersets of an algebraically closed set
   $Z=\acl(Z)$ in $\cps_k$. We have that \[X\ind^{\cps_k}_Z Y\] if and
   only if
\begin{itemize}
\item $X\ind^{\ps}_Z Y$, and
\item $\EP(\fcl X)\cap \EP(\fcl Y)\subset Z$.
\end{itemize}
Types over algebraically closed sets are stationary, that is, the
theory $\cps_k$ has weak elimination of imaginaries.
\end{cor}

\begin{proof}
  Since $\ps$ has weak elimination of imaginaries, we have that
  non-forking in $\cps_k$ implies nonforking in the reduct $\ps$ over
  algebraically closed sets, by \cite[Lemme 2.1]{BMPW15}. Clearly
  $\EP(\fcl X)\cap \EP(\fcl Y)\subset Z$.

  \noindent For the other direction, we may assume that $X=\fcl X$ and
  $Y=\fcl Y$. Lemma \ref{L:AP_unabh} yields that
  \[ X \cup \EP(X)\ind^{\ps}_Z Y \cup \EP(Y).\]  Since $\acl(X)=X \cup
  \EP(X)$, Remark \ref{R:ver_acl} implies that the set $\acl(X) \cup
  \acl(Y)$ is algebraically closed in $\ps$. We need only show that it
  contains all exceptional points, so it determines a unique type in
  the stable theory $\cps_k$. If $c$ is an exceptional point of a
  plane $a$ and a line $b$ in $\acl(X)\cup \acl(Y)$, we may assume
  that $a$ lies in $X$ and $b$ lies in $Y$. Since $a$ and $b$ are
  directly connected and $X \ind^{\ps}_Z Y$, Remark \ref{R:ver_acl}
  implies that $a$ or $b$ lies in $Z$. Therefore $c$ lies in
  $\EP(X)\cup\EP(Y)$ and hence is contained in $\acl(X)\cup \acl(Y)$,
  as desired.
\end{proof}

\begin{cor}\label{C:forking_acl}
  Let $X$, $Y$ and $Z=\acl(Z)$ be sets such that \[X\ind_Z Y.\] Then
  $\fcl{X,Y}\cap\acl(X,Z)=\fcl{X,Y}\cap\fcl{X,Z}$.
\end{cor}
\begin{proof}
  Let $\xi$ be in $\fcl{X,Y}\cap\acl(X,Z)$. The independence \[X\ind_Z
  Y\] yields that \[\xi , X \ind_Z Y.\] It follows from Corollary
  \ref{C:color_Forking} that
  \[ \xi , X \ind^{\ps}_Z Y,\]
  and thus \[ \xi \ind^{\ps}_{X, Z} X,Y.\] Since $\xi$ lies in
  $\fcl{X,Y}$, the above independence implies that $\xi$ lies in
  $\fcl{X,Z}$, as desired.
\end{proof}
\begin{prop}\label{P:ext_Iso}
  Let $X=\fcl X$ and $Y=\fcl Y$ be two subsets of a model of $\cps_k$.
  A map $F:X \to Y$ is elementary with respect to the theory $\cps_k$
  if and only if it satisfies the following conditions:
  \begin{enumerate}
  \item The map $F$ is a partial isomorphism with respect to the
    reduct $\ps$.
  \item The function $F$ preserves colours of lines and sections.
  \item\label{Cond:iso} For all $a$, $a'$ and $b$ in $X$, we have that
    $\expt(a,b)=\expt(a',b)$ if and only if
    $\expt(F(a),F(b))=\expt(F(a'),F(b))$.
\end{enumerate}
\end{prop}

\begin{proof}
  We need only show that $F$ is elementary, if it satisfies all three
  conditions. By Theorem \ref{T:CPSk_vollst}, it suffices to show that
  $F$ extends to a partial isomorphism $\tilde F$ preserving colours
  between $\acl(X)=X \cup \EP( X)$ and $\acl(Y)=Y \cup \EP( Y)$.

  For each line $b$ in $X$ contained in a plane $a$ of $X$, set
  $\tilde F (\expt(a,b))=\expt(F(a), F(b))$. Let us first show that
  $\tilde F$ is well-defined, which analogously yields that $\tilde F$
  is a bijection. Suppose that $\expt(a,b)=\expt(a_1, b_1)$, for a
  line $b_1$ contained in the plane $a_1$, both in $X$. If $b\neq
  b_1$, then $\expt(a,b)$ is the unique intersection of $b$ and $b_1$,
  both lines in $X$, so $\expt(a,b)$ lies in $X$ and hence its image
  is determined by $F$. Otherwise, we conclude that $b=b_1$, and thus
  $\tilde F$ is bijective, by Condition $(\ref{Cond:iso})$.

  Similarly, the map $\tilde F $ defined above is a partial
  isomorphism with respect to the reduct $\ps$. We need only show that
  $\tilde F$ preserves the colours of sections. Choose a new point
  $\expt(a,b)$ not in $X$ and an arbitrary plane $a_1\neq a$ in $X$
  containing $\expt(a,b)$. Since $\expt(a,b)$ does not lie in $X$, the
  intersection of $a$ and $a_1$ cannot solely consist of the point
  $\expt(a,b)$. Hence, the intersection of $a$ and $a_1$ is given by a
  unique line $b_1$, which lies in $X$ and contains $\expt(a, b)$. We
  conclude as before that $b=b_1$. The colour of $\expt(a,b)$ in $a_1$
  is uniquely determined according to whether $\expt(a,
  b)=\expt(a_1,b)$, and thus so is the colour of its image in $F(a_1)$
  by $\tilde F$, by Condition $(\ref{Cond:iso})$.

\end{proof}
\section{Colored paths}\label{S:ColPath}

We will now show that the theory $\cps_k$ is not \icyc, and hence it is not
equational, yet every proper cycle of types has length at least $k$
(cf. Theorem \ref{T:path_atleast_k}), so we expect the complexity of
these theories to increase as $k$ grows. However, we do not know
whether two of these theories are bi-interpretable.

\begin{theorem}\label{T:path_k}
  In $\cps_k$ there is a proper cycle of types \[ p_0(x;y)\pf
  p_1(x;y)\pf \ldots \pf p_{k-1}(x;y)\pf p_0(x;y),\]
  \noindent where both the variables $x$ and $y$ have length $1$. In
  particular, the theory $\cps_k$ is not equational.
\end{theorem}

\begin{proof}
  For each $0\leq r<k$, a pair $(a, c)$ with colour $I_r$ has a unique
  type $p_r=\tp(c,a)$, for the set $\{a, c\}$ is algebraically closed,
  since it is the intersection of all the flags containing $a$ and
  $c$, and it is closed under exceptional points, for it contains no
  line. Clearly $p_r\neq p_{r+1}$, for each $0\leq r<k$.

  It suffices to show that $p_r\pf p_{r+1}$: Let $(a, c)$ with colour
  $I_r$, and choose a line $b$ connecting them with colour $r$. Let
  $c'$ be the exceptional point of $b$ in $a$, so $(c',a) \models
  p_{r+1}$. Now, the set $\{b\}$ is algebraically closed in $\cps_k$.
  By Corollary \ref{C:color_Forking}, the points $c$ and $c'$ have the
  same strong type over $b$, and
  \[ c\ind_b a. \]
  Corollary \ref{C:Kern_Pf} implies that $p_r=\tp(c,a) \pf
  \tp(c',a)=p_{r+1}$, as desired.
 \end{proof}

 \begin{theorem}\label{T:path_atleast_k}
   Let $x$ and $y$ be finite tuples of variables. In $\cps_k$, every
   proper cycle of types \[ p_0(x;y)\pf p_1(x;y)\pf \ldots \pf
   p_{n-1}(x;y)\pf p_0(x;y),\]
   \noindent has length $n\geq k$.
 \end{theorem}

 \begin{proof}
   A proper cycle of types $p_0(x;y)\pf p_1(x;y)\pf \ldots \pf
   p_{n-1}(x;y)\pf p_0(x;y)$ as above induces a cycle in the reduct
   $\ps$, which is equational. Therefore, the colourless reducts of
   $p_r$ and $p_s$ agree, for all $r$, $s$.

   Corollary \ref{C:Kern_Pf} implies that there are tuples $f$,
   $e_0,\ldots, e_n$ and algebraically closed subsets $Z_0, \ldots,
   Z_{n-1}$ such that:
   \begin{itemize}
   \item $\models p_r(e_r,f)$, for $0\leq r\leq n-1$, and
     $\models p_0(e_n,f)$.
   \item $e_r\equiv_{Z_r} e_{r+1}$ for $0\leq r\leq n-1$.
   \item $e_r\ind_{Z_r} f$ for $0\leq r\leq n-1$.
   \end{itemize}

   Set $Y=\acl(f)$, $X_r=\acl (e_r)$, for $0\leq r\leq n$. Since the
   definable and algebraic closure coincide, and the colourless
   reducts of all $p_r$ agree, all the types $\tp^{\ps}(X_rY)$ are
   equal. Denote $\fcl {X_r Y}$ by $P_r$. We find colourless
   isomorphisms \[F_r: P_r\to P_{r+1}, \] which fix $Y$ pointwise.
   Note that $X_r$ and $X_{r+1}$ have the same type over $Z_r$, for
   $r\leq n-1$. Lemma \ref{L:pairwise} yields that $\tp^\ps(X_rYZ_r)
   =\tp^\ps(X_{r+1}YZ_r)$, for $r\leq n-1$. The above map $F_r$
   extends to a colourless isomorphism between $\fcl{X_rYZ_r}$ and
   $\fcl{X_{r+1}YZ_r}$, which is the identity on $\fcl{YZ_r}$. We will
   still refer to this colourless isomorphism as $F_r$, keeping in
   mind that it is elementary in the sense of $\cps_k$ on
   $\fcl{X_rZ_r}$ and (clearly) on $\fcl{YZ_r}$ separately. Observe
   that
   \[\fcl{X_rYZ_r}=\fcl{X_rZ_r}\cup \fcl{YZ_r},\]
   by the Remark \ref{R:ver_acl} (\ref{R:ver_acl:acl}).

   If a set $W$ is finite, so are the closures $\fcl{W}$ and
   $\acl(W)$. Define its \emph{defect} as the natural number
   \[\de(W)=|\acl(W)\setminus \fcl{W}|= | \EP(\fcl W) \setminus \fcl W|.\]

   \begin{claimstern}
     For each $r\leq n-1$, we have that $\de(P_r)\geq \de(P_{r+1})$.
   \end{claimstern}
   \begin{claimsternproof}
     Whenever $\expt(a,b)=\expt(a_1,b_1)$, for $b$ and $b_1$ in $P_r$,
     with $b\neq b_1$, then the point $\expt(a,b)$ lies in $P_r$ by
     \ref{R:ver_acl} (\ref{R:ver_acl:inf}). Thus, it suffices to show
     the following:
     \begin{enumerate}
     \item\label{C:1st} Whenever the line $b$ in $P_r$ lies in the
       plane $a$ in $P_r$, with $\expt(a,b)$ in $P_r$, then
       $\expt(F_r(a), F_r(b))$ lies in $P_{r+1}$.
     \item\label{C:2nd} Whenever $a$, $a_1$ and $b$ lie in $P_r$ and
       $\expt(a,b)=\expt(a_1,b)$, then
       $\expt(F_r(a),F_r(b))=\expt(F_r(a_1),F_r(b))$.
     \end{enumerate}
     For $(\ref{C:1st})$, since $X_r \ind_{Z_r} Y$, the plane $a$ and
     the line $b$ must both lie in the same set $P_r\cap \fcl{X_rZ_r}$
     or in $P_r\cap \fcl{YZ_r}$, by the independence
     \[a\ind^{\ps}_{Z_r} b\]
     \noindent and Remark \ref{R:ver_acl} (\ref{R:ver_acl:direct}).
     For example, let $a$ and $b$ lie in $P_r\cap \fcl{X_rZ_r}$, so
     $\expt(a,b)$ lies in $P_r\cap\acl(X_rZ_r)=P_r\cap \fcl{X_rZ_r}$,
     by Corollary \ref{C:forking_acl}. Since $F_r$ is elementary on
     $P_r\cap \fcl{X_rZ_r}$, we have that $\expt(F_r(a),
     F_r(b))=F_r(\expt(a,b))$ lies in $P_{r+1}$, as desired. Observe
     that we have actually shown that \[\expt(a,b) \in P_r
     \Longleftrightarrow \expt(F_r(a), F_r(b))\in P_{r+1}.\]

     \noindent For $(\ref{C:2nd})$, we need only consider the case
     when $a\neq a_1$ and the exceptional point
     $\expt(a,b)=\expt(a_1,b)$ does not lie in $P$, by
     $(\ref{C:1st})$. Again, if both $a$ and $a_1$ lie in $
     \fcl{X_rZ_r}$ or in $\fcl{YZ_r}$, then so does $b$, and we are
     done by Proposition \ref{P:ext_Iso}, since $F_r$ is elementary on
     each side. If this is not the case, and $a$ lies in $P_r\cap
     \fcl{X_rZ_r}$ and $a_1$ in $P_r\cap \fcl{YZ_r}$, then the line
     $b$ lies in $P_r\cap Z_r$, by the Remark \ref{R:ver_acl}
     (\ref{R:ver_acl:pp}). Thus the point $\expt(a,b)=\expt(a_1,b)$
     lies in $\acl(X_r, Z_r)\cap \acl(Y,Z_r)=Z_r$, so we conclude as
     before since $F_r$ is elementary on each side separately.
   \end{claimsternproof}

   As $P_0$ and $P_n$ have the same type, their defect is the same, so
   $\de(P_r)=\de(P_{r+1})$, for all $0\leq r\leq n-1$. Hence, for all
   $a$, $a'$ and $b$ in $P_r$, we have that
   \[ \expt(a,b)=\expt(a',b)\text{  if and only if }
   \expt(F_r(a),F_r(b))=\expt(F_r(a'),F_r(b)).\]

   Since $P_r$ and $P_{r+1}$ are closed in the reduct $\ps_k$, but
   $\tp(P_r)\neq \tp(P_{r+1})$, Proposition \ref{P:ext_Iso} implies
   that $F_r$ restricted to $P_r$ cannot preserve colours. As $F_r$ is
   elementary on each side separately, the colours of lines are
   preserved. Thus, there is a pair $(a,c)$ in $P_r$ whose colour $j$,
   with $0\leq j<k$, is not preserved under $F_r$. We will show now
   that the colour of the pair $(F_r(a), F_r(c))$ is $j+1$.

   Since $F_r$ is elementary on $\fcl{X_rZ_r}$ and on $\fcl{YZ_r}$
   separately, neither $a$ nor $c$ lie in $Z_r$. The independence
   $X_r\ind^{\ps}_{Z_r} Y$ and Remark \ref{R:ver_acl}
   (\ref{R:ver_acl:pp}) yield that there is a line $b$ in $Z_r$
   connecting $a$ and $c$. The characterisation of the independence in
   Corollary \ref{C:color_Forking} implies that $c\neq \expt(a,b)$.
   Hence the line $b$ must have colour $j$. The map $F_r$ is the
   identity on $Z_r$, and the plane $F_r(a)$ is connected to the point
   $F_r(c)$ by $b=F_r(b)$, so the only possible colours for the pair
   $(F_r(a),F_r(c))$ are $j$ or $j+1$. As the colour of the pair
   $(a,c)$ is not preserved, we deduce that $(F_r(a),F_r(c))$ has
   colour $j+1$, as desired.

   Let $F_n$ be the $\cps_k$-elementary map mapping $P_n$ to $P_0$ (as
   both $(e_0, f)$ and $(e_n,f)$ realise the type $p_0$) and write
   $F^r= F_r\circ \ldots\circ F_0$. Notice that the map $F^n$ is the
   identity of $P_0$. Let $(a,c)$ be one of the pairs in $P_0$ whose
   colour $j_0$ changes under $F_0$. The colours of the pairs \[
   (a,c), F^0(a,c), \ldots, F^{n-1}(a,c) \] change at each step by at
   most adding $1$ (modulo $k$), so the colour of $F^{n-1}(a,c)$
   equals $j_0+m$ modulo $k$, for some $1\leq m\leq n$. Since $F_n$
   preserves colours and $F^{n}(a,c)=(a,c)$, we have that $m$ is
   divisible by $k$, and thus $k\leq m\leq n$. We conclude that the
   original cycle had length at least $k$.
 \end{proof}

\begin{remark}\label{R:Bij_ohneZykeln}
  Given a function $\pi:\{0,\ldots, k-1\} \to \{0,\ldots, k-1\}$ with
  no fixed points, we could similarly consider the theory $\cps_\pi$
  of
  colored pseudospaces such that given a line $b$ with colour $i$
  inside a plane $a$, all points in $b$ lie in the section $I_i(a)$
  except one unique exceptional point which lies in $I_{\pi(i)}(a)$.

  The corresponding theory $\cps_\pi$ is not equational. Every closed
  path of real types has length at least the length of the shortest
  $\pi$-cycle.
\end{remark}


\end{document}